\documentclass[11pt]{amsart}
\usepackage{fullpage}
\usepackage{amsmath}
\usepackage{amsfonts}
\usepackage{amssymb}
\usepackage{amsthm}
\usepackage{graphicx,xy}
\usepackage{enumerate}
\usepackage{latexsym}

\DeclareMathAlphabet{\mathpzc}{OT1}{pzc}{m}{it}

\def\cal{\mathcal}
\def\d{{\mathrm d}}
\def\F2{\mathbb F_2}
\def\A2{{\mathcal A}_2}

\def\SC{{\cal{SC}}}
\def\hsc{{\mathrm{\bf{sc}}}}

\def\Lie{{\cal{L}ie}}
\def\Com{{\cal{C}om}}
\def\Ass{{\cal{A}ss}}
\def\Ger{{\cal{G}er}}
\def\SCvor{{\cal{SC}}^{\rm vor}}
\def\hscvor{{\mathrm{\bf{sc}}^{\rm vor}}}

\def\dgvs{{\bf dgvs}}
\def\Der{{\rm Der}}

\def\Hom{{\rm Hom}}
\def\hom{{\rm hom}}
\def\End{{\rm End}}
\def\sgn{{\rm sgn}}
\def\ide{{\rm id}}

\def\cl{\mathpzc c}
\def\op{\mathpzc o}

\def\kfield{k}

\newcommand{\ac}{\scriptstyle \text{\rm !`}}
\newcommand{\epi}{\twoheadrightarrow}

\newcommand{\wiggly}{\xymatrix@1@C=15pt{  && \ar@{~}[ll]}}

\newcommand{\straight}{\xymatrix@1@C=30pt{  & \ar@{-}[l]}}

\newtheorem{thm}{Theorem}[subsection]
\newtheorem{lem}[thm]{Lemma}
\newtheorem{prop}[thm]{Proposition}
\newtheorem{cor}[thm]{Corollary}
\theoremstyle{definition}
\newtheorem{defn}[thm]{Definition}

\theoremstyle{remark}

\theoremstyle{remark}

\xyoption{all}
\CompileMatrices

\begin{document}

\title[On the spectral sequence of the Swiss-cheese operad]{On the spectral sequence of the Swiss-cheese operad}
\author{Eduardo Hoefel}
\address{Universidade Federal do Paran\'a, Departamento de Matem\'atica C.P. 019081, 81531-990 Curitiba, PR - Brazil }
\email{hoefel@ufpr.br}
\author{Muriel Livernet}
\address{Universit\'e Paris 13, Sorbonne Paris Cit\'e,  LAGA, CNRS (UMR 7539), 99 avenue
  Jean-Baptiste Cl\'ement, F-93430 
  Villetaneuse, France}
\email{livernet@math.univ-paris13.fr}
\thanks{This collaboration is funded by the MathAmSud program ``OPECSHA'' coordinated by M. Ronco, by Funda\c c\~ao Arauc\'aria (Brazil) grant FA 490/6032 and by the r\'eseau Franco-br\'esilien de math\'ematiques.}
\date{\today}
\begin{abstract}
We prove that the homology of the Swiss-cheese operad is a Koszul operad. As a consequence, we obtain that the spectral sequence associated to the stratification of the compactification of points on the upper half plane collapses at the second stage, proving a conjecture by A. Voronov in \cite{Voronov99}. However, we prove that the operad obtained at the second stage differs from the homology of the Swiss-cheese operad.
\end{abstract}
\maketitle

%
%
%
%


\section{Introduction}

An operad $\cal O$ is called {\it differentiable} when it is defined on the symmetric monoidal category of differentiable manifolds.  
If each $\cal O(n)$ is a manifold with corners whose connected boundary components are cartesian products of $\cal O(k)$ (with $k < n$) and the operad structure is given by the inclusion map of the boundary strata, then the operad is called {\it stratified}. To every stratified operad is associated a dg-operad 
given by the spectral sequences induced by a natural filtration on its singular chain complex: the boundary strata filtration. Since that filtration is given 
by the codimension of the boundary components, it is finite and hence converges to the homology $H_*(\cal O)$ at some finite stage. One would naturally 
wonder wether the operad structure induced on homology by the spectral sequence coincides with the one induced by the topological operad structure.

In the present paper we study the spectral sequence of a stratified operad, given by Kontsevich's compactification \cite{kont:defquant-pub}, which is homotopically equivalent to the Swiss-cheese operad $\SC$. We show that $H_*(\SC)$ is an inhomogeneous Koszul operad. We also show that the operad structure induced on $H_*(\SC)$ by the spectral sequence of its stratified version differs from the structure given by the topological operad structure of $\SC$. The difference between the two operad structures in $H_*(\SC)$ is explored in terms of the Koszul duality theory for inhomogeneous quadratic operads developed by Galvez-Carillon, Tonks and Vallette in \cite{GTV12}.

Among the related algebraic structures are Kajiura and Stasheff's OCHA \cite{KS06a,KS06c}, Leibniz pairs \cite{FGV95} and extensions of those considered by Dolgushev \cite{Dolgushev11}. 
The relation between OCHAS, Leibniz pairs and the Swiss-cheese operad has been carefully studied by the authors in \cite{HoeLiv12},
where the zero-th homology of the Swiss-cheese operad $\SC$ were related to the first row of the spectral sequence associated to
the Kontsevich compactification. One of the purpose of \cite{HoeLiv12} was to prove an $\SC$ analogue of the following fact concerning the little disks operad $\cal D_2$:

\medskip
\noindent{\bf Proposition.} \it
 The zero-th homology of the operad $\cal D_2$ is Koszul dual to a suspension of the top homology.\rm

\medskip
In the case of little discs, the zero-th homology is the operad $\Com$ and the top homology is a desuspension of the operad $\Lie$. In fact, the above Proposition 
is a consequence of a theorem, proved by Getzler and Jones, according to which the Gerstenhaber operad $H_*(\cal D_2)$ is, up to suspension, a self-dual Koszul operad.

We proved in \cite{HoeLiv12} that the zero-th homology of $\SC$ is a Koszul quadratic-linear operad, and that its Koszul dual $H_0(\SC)^!$, which is a dg-operad, has for homology a suspension of the top homology. Note that in the context of $\SC$, the top homology does not form an operad, so by top homology we mean the smallest operad containing the top degrees generators.

The little disks operad $\cal D_2$ is not stratified. However, by considering the real Fulton-MacPherson compactification of the moduli space
of points in the complex plane, we get a homotopically equivalent stratified operad sometimes denoted $\cal F_2$ (see \cite{Salvatore01}). The same compactification procedure can be applied to the Swiss-cheese operad $\SC$ (the homotopy equivalent stratified operad obtained is sometimes denoted $\cal H_2$).
So, by passing to a homotopy equivalent operad, we can assume that both little disks and Swiss-cheese operads are stratified. 
We note that from the cofibrancy of $\cal F_2$ proven by Salvatore \cite{Salvatore01}, the homotopy equivalence between 
the above operads are compatible with the operad structure, hence the topological operad structure of $\cal D_2$ and $\cal F_2$ (resp. $\SC$ and $\cal H_2$)
induce the same operad structure on homology. Hence, to avoid cumbersome notation we will work with the stratified versions of little disks and Swiss-cheese,  while keeping the notation: $\cal D_2$ and $\SC$.

The main result of this paper says that the homology of the Swiss-cheese operad is a quadratic-linear Koszul operad in the sense of \cite{GTV12}. The same result is true for the 
homology of $\SCvor$, a variant of $\SC$. As an application, in Theorem \ref{T:ss} we show that the spectral sequence $E(SC)$ of the Swiss-cheese operad collapses at the second stage, proving a conjecture by A. Voronov in \cite{Voronov99}. The relation (modulo (de)suspension) between $E^1(SC)$ and the cobar construction of the cohomology cooperad $H^*(\SC)$ is well known \cite{Dolgushev11}, but in our setting it is slightly different, so a proof is given in Lemma \ref{lemma:cobar}. Finally, using that Lemma and the Koszul duality theory in the quadratic-linear framework, we prove that the spectral sequence given by the stratified structure of $\SC$ induces on $H_*(\SC)$ an operad structure that is different from the one induced by the topological operad structure of $\SC$, while the $\mathbb S$-module structure is the same. 

Getzler and Jones \cite{GetJon94} have proven that 
$
 E^2(\cal D_2) = \Lambda^{-1}(H_*(\cal D_2)^!)$, where $\Lambda^{-1}(H_*(\cal D_2)^!)$ is  isomorphic to $\mathrm{{\bf e}}_2 = H_*(\cal D_2)$.
In analogy to their result, and recalling that $H_*(\SC)^!$ is a dg-operad, in the present paper we prove that
\begin{equation}\label{eq:hoeliv}
 E^2(\SC) = H(\Lambda_{\cl}^{-1}(H_*(\SC)^!)) = H(\Lambda_{\cl}^{-1}(\hsc^!)),
\end{equation} 
where $H(\Lambda_{\cl}^{-1}(\hsc^!))$ is isomorphic to $\hsc = H_*(\SC)$ as an $\mathbb S$-module, but not as an operad.

%
%

%
%

\medskip
The plan of the paper is the following. The second section is devoted to preliminaries and notations. The third section is devoted to the homology of the operads $\SCvor$ and $\SC$. 
As in \cite{HoeLiv12}, in order to understand the structure of the operad 
$\hsc=H_*(\SC)$, it is necessary to first understand  the operad $\SCvor$, another version of the Swiss-cheese operad, whose homology is quadratic. The end of  section 3 is devoted to the structure of $H_*(\hsc^!)$. We use the technics of distributive laws, as well as the results obtained in \cite{HoeLiv12}.
The fourth section concentrates on the spectral sequence of $\SC$.

\section{Preliminaries}

\subsection{On differential graded vector spaces.} We work on a ground field $\kfield$ of characteristic $0$.
 The category  $\dgvs$ is the category of lower graded $\kfield$-vector spaces together with a differential of degree $-1$.  Objects in $\dgvs$ are called for short  dgvs.  The degree of $x\in V$, where $V$ is a dgvs is  denoted by $|x|$. We say that a dgvs $V$ is finite dimensional if for each $n$, the vector space $V_n$ is finite dimensional.
 
 The vector space $\hom_\kfield(V,W)$ denotes the $\kfield$-linear morphisms between two vector spaces  $V$ and $W$. When $V$ and $W$ are objects in $\dgvs$,
 the differential graded vector space of maps from $V$ to $W$ is
 $\bigoplus_{i \in \mathbb Z} \Hom_i(V,W)$, where $\Hom_i(V,W) = \prod_n \hom_\kfield(V_n, W_{n+i})$ together with the differential
$(\partial f)(v)=\d_W(f(v))-(-1)^{|f|}f(d_Vv)$.

 The graded linear dual of $V$ in $\dgvs$ is  $V^*=\Hom(V,\kfield)$, where $\kfield$ is concentrated in degree $0$ with $0$-differential. Consequently $(V^*)_n=(V_{-n})^*$ and $(\partial f)(x)=-(-1)^n f(d_Vx)$ for any $f\in (V^*)_n$ and 
  $x\in V_{-n+1}$.

  The suspension of a dgvs $V$ is denoted by $sV$ and defined as $(sV)_n=V_{n-1}$.

\subsection{On operads, 2-colored operads, cooperads}

\subsubsection{On the symmetric group}
The symmetric group acting on $n$ elements is denoted by $S_n$.  An element $\sigma\in S_n$ will be denoted by its image notation $(\sigma(1)\ldots\sigma(n))$. The trivial representation of $S_n$ is denoted by $\kfield$, the signature representation by $\sgn_n$ and the regular representation by $\kfield[S_n]$.

\subsubsection{Collections and $\mathbb S$-modules} 
In this article, we  work with 2-colored (co)-operads, either in the category of spaces, or in $\dgvs$. The  colors we consider are denoted by $\cl$ (for closed) and $\op$ (for open). 
A {\sl $2$-collection}  $\cal P$ is a  family of dgvs, 
$(\cal P(\underline c;d)=\cal P(c_1,\ldots,c_n,d))_{\{c_i,d\in\{\cl,\op\}\}}$.
Let $\underline c=(c_1,\ldots,c_n)$ be an $n$-tuple of colors. The symmetric group $S_n$ acts on the set of $n$-tuples of colors by
$\underline c\cdot \sigma=(c_{\sigma(1)},\ldots,c_{\sigma(n)} ).$

An {\sl $\mathbb S$-module} is 
a $2$-collection  $\cal P$ together with an action of the symmetric groups,
sending $(x\in \cal P(c_1,\ldots,c_n;d),\sigma\in S_n)$ to $x\cdot\sigma \in \cal P(\underline c\cdot \sigma;d)$.

Note that it is an extension of the usual definition of an $\mathbb S$-module, which is a family of dgvs
 $(\cal Q(n))_{n\geq 1}$ such that for each $n$,  $\cal Q(n)$ is a right $S_n$-module. We can consider this collection as an $\mathbb S$-module, where
 $\cal Q(\underline c;d)$ is $\cal Q(n)$ if $\forall i, c_i=\cl$ and $d=\cl$ and is $0$ otherwise.

Given an $\mathbb S$-module $M$, we may want to consider some truncation $N$ of it, invariant under the action of the symmetric groups. 
By definition a sub-$\mathbb S$-module of $N$ is a sub-dgvs invariant under the induced action  of the symmetric groups.

\subsubsection{Operads}
 The $2$-collection $I$ defined by $I(\cl;\cl)=\kfield, I(\op;\op)=\kfield$ and $I(\underline c;d)=0$ elsewhere, plays a special role. Indeed,
a {\sl 2-colored operad} is an $\mathbb S$-module together with a unit map $\eta:I\rightarrow \cal P$ and
composition maps 
$$\gamma: \cal P(c_1,\ldots,c_n;d)\otimes\cal P(\underline b^1;c_1)\otimes\ldots\otimes\cal P(\underline b^n;c_n)\rightarrow
\cal P(\underline b^1,\ldots,\underline b^n;d),$$
which are associative, unital and respects the action of the symmetric groups.

We write $f(g_1,\ldots,g_n)$ for the image of $f\otimes g_1\otimes\ldots\otimes g_n$ or
$f(\ide^{\otimes i}\otimes g\otimes \ide^{\otimes n-1-i})$ whenever every $g$ except one is the identity.

We often use the same notation for $f$ in $\cal P$ or for $f$ seen as an operation on variables. In that context, we use the  Koszul sign convention

$$(f\otimes g)(\underline a\otimes \underline b)=(-1)^{|a||g|} f(a)\otimes g(b).$$

Because of the action of the symmetric groups, one may only consider  the spaces
$$\cal P(n,m;d)=\cal P(\underbrace{\cl,\ldots\cl}_{n\ \hbox{\rm times}},\underbrace{\op,\ldots\op}_{m\ \hbox{\rm times}};d).$$

In this paper, we only consider $2$-colored operads such that  $\cal P(0,0;x)=0$  and 
$\cal P(1,0;\cl)=\kfield=\cal P(0,1;\op)$.  They are naturally {\sl augmented}, that is, there is a morphism of operads
$\cal P\rightarrow I$ and $\overline{\cal P}$  denotes the kernel of this map.

Any operad $\cal P$ can be considered as a $2$-colored operad with
$\cal P(\underline c;d)=\cal P(n)$ if $\forall i, c_i=\cl$ and $d=\cl$, $\cal P(\op;\op)=\kfield$ and  $\cal P(\underline c;d)=0$ otherwise.

In the sequel, we often use the generic terminology of operads for either operads or 2-colored operads, or operads seen as $2$-colored operads.

%
%

\subsubsection{Suspension of\;  $\mathbb S$-modules and operads}
The suspension of the $\mathbb S$-module $\cal P$ is
$$\Lambda\cal P(n,m;x)=s^{1-n-m}\cal P(n,m;x)\otimes \sgn_{n+m}.$$
If $\cal P$ is an operad, then the structure of $\cal P$-algebra on the pair $(V_\cl,V_\op)$ is equivalent to the structure of $\Lambda\cal P$-algebra on the pair $(sV_\cl,sV_\op)$.

 The suspension of the $2$-collection $\cal P$ with respect to the color $\cl$ is
 $$\Lambda_\cl\cal P(n,m;x)=s^{\delta_{x,\cl}-n}\cal P(n,m;x)\otimes \sgn_{n},$$
 where $\delta$ denotes the Kronecker symbol.
If $\cal P$ is an operad, then the structure of $\cal P$-algebra on the pair $(V_\cl,V_\op)$ is equivalent to the structure of $\Lambda_\cl\cal P$-algebra on the pair $(sV_\cl,V_\op)$.

\subsubsection{Operads defined by generators and relations}

The free  operad  generated by an $\mathbb S$-module $E$ is denoted by $\cal F(E)$. It is weight graded by the number $n$ of vertices of the underlying trees and $\cal F^{(n)}(E)$ denotes the component of weight $n$.

A {\sl quadratic  operad} $\cal F(E,R)$  is an operad of the form  $\cal F(E)/(R)$, where $E$ is an $\mathbb S$-module, $R$ is a
 sub-$\mathbb S$-module of $\cal F^{(2)}(E)$  and $(R)$ is the ideal generated by $R$. There are analogous notions of  cooperads,  free  cooperads $\cal F^c(E)$ cogenerated by $E$, and of cooperads cogenerated by an $\mathbb S$-module $V$  with correlation $R$ denoted by $C(V,R)$.

Describing an operad is equivalent to describing algebras over it. In the text, we say that an operad $\cal P$ is generated by $E$ with relations $R$ written as$$(*)\quad r_1=r_2.$$ 
This notation means that any $\cal P$-algebra satisfies the relation $(*)$. At the level of operads, this is understood as  
$R$ contains the element $r_1-r_2$.

\subsubsection{Koszul dual}\label{S:Koszul} Any quadratic operad $\cal P = \cal F(E,R)$ admits a {\sl Koszul dual cooperad} given by $\cal P^{\ac}=C(sE,s^2R).$

The {\sl Koszul dual operad} $\cal P^!$ of a finite dimensional quadratic operad $\cal P$ is
\begin{equation}\label{E:Koszuldual}
\cal P^!:=(\Lambda \cal P^{\ac})^* \text{ or equivalently }  \cal P^{\ac}=(\Lambda \cal P^!)^*.
\end{equation}

When $\cal P$ is a binary quadratic  operad,
we can use the original definition of Ginzburg and Kapranov \cite{GinKap94} (see also \cite[chapter 7]{LodVal}) to compute its Koszul dual operad. Namely, if $\cal P=\cal F(E,R)$, then
$\cal P^!=\cal F(E^\vee,R^{\bot})$, where $E^\vee=E^*\otimes\sgn_2$ and $R^{\bot}$ denotes the orthogonal of $R$ under the induced pairing $\cal F^{(2)}(E)\otimes \cal F^{(2)}(E^\vee)\rightarrow \kfield$.

%
%

\subsection{Two versions of the Swiss-cheese operad}
Here we recall the two definitions for the Swiss-cheese operad we have introduced in \cite{HoeLiv12}. 
We denote by $\cal D_2$ the little disks operad.

For $m,n \geqslant 0$ such that $m + n >0 $, let us define $\cal{SC}(n,m;\op)$ as the space of those 
configurations $d \in \cal (2n + m)$ such that its image in the disk $D^2$ is invariant under complex conjugation 
and exactly $m$ little disks are left fixed by conjugation.
A little disk that is fixed by conjugation must be centered at the real line, in this case
it is called {\it open}. Otherwise, it is called {\it closed}.  
The little disks in $\cal{SC}(n,m;\op)$ are labelled according to the following rules:
\begin{enumerate}[i)]
\item Open disks have labels in $\{1, \dots, m \}$ and closed disks 
have labels in $\{ 1, \dots, 2n \}$. 
\item Closed disks in the upper half plane have labels in $\{ 1, \dots, n \}$. 
If conjugation interchanges the images of two closed disks, their labels must be congruent modulo $n$.
\end{enumerate}

There is an action of $S_n \times S_m$ on $\cal{SC}(n,m;\op)$ extending the action of
$S_n \times \{ e \}$ on pairs of closed disks having modulo $n$ congruent labels and the action of
$\{ e \} \times S_m$ on open disks. Figure \ref{Swiss_disc} illustrates a point in the space 
$\cal{SC}(n,m;\op)$.

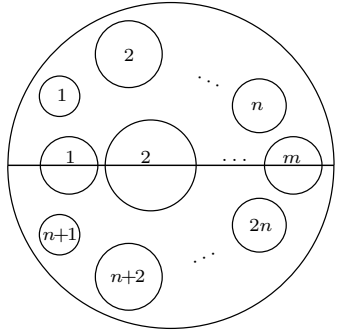
\begin{figure}
\begin{picture}(0,0)%
\includegraphics{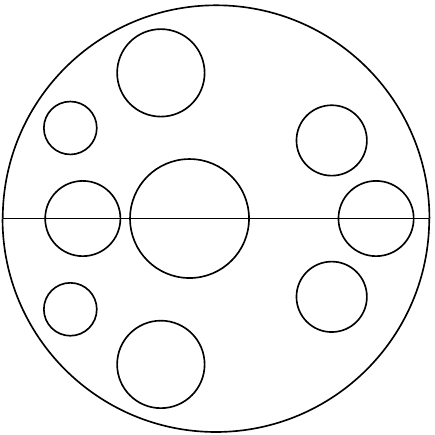}%
\end{picture}%
\setlength{\unitlength}{4144sp}%
\begingroup\makeatletter\ifx\SetFigFont\undefined%
\gdef\SetFigFont#1#2#3#4#5{%
  \reset@font\fontsize{#1}{#2pt}%
  \fontfamily{#3}\fontseries{#4}\fontshape{#5}%
  \selectfont}%
\fi\endgroup%
\begin{picture}(1976,1970)(5195,-5068)
\put(5500,-3700){\makebox(0,0)[lb]{\smash{{\SetFigFont{7}{8.4}{\rmdefault}{\mddefault}{\updefault}$1$}}}}
\put(5900,-3458){\makebox(0,0)[lb]{\smash{{\SetFigFont{7}{8.4}{\rmdefault}{\mddefault}{\updefault}$2$}}}}
\put(6660,-3750){\makebox(0,0)[lb]{\smash{{\SetFigFont{7}{8.4}{\rmdefault}{\mddefault}{\updefault}$n$}}}}
\put(5550,-4068){\makebox(0,0)[lb]{\smash{{\SetFigFont{7}{8.4}{\rmdefault}{\mddefault}{\updefault}$1$}}}}
\put(6000,-4068){\makebox(0,0)[lb]{\smash{{\SetFigFont{7}{8.4}{\rmdefault}{\mddefault}{\updefault}$2$}}}}
\put(6850,-4068){\makebox(0,0)[lb]{\smash{{\SetFigFont{7}{8.4}{\rmdefault}{\mddefault}{\updefault}$m$}}}}
\put(5415,-4531){\makebox(0,0)[lb]{\smash{{\SetFigFont{7}{8.4}{\rmdefault}{\mddefault}{\updefault}$n\!\!+\!\!1$}}}}
\put(5820,-4782){\makebox(0,0)[lb]{\smash{{\SetFigFont{7}{8.4}{\rmdefault}{\mddefault}{\updefault}$n\!\!+\!\!2$}}}}
\put(6650,-4474){\makebox(0,0)[lb]{\smash{{\SetFigFont{7}{8.4}{\rmdefault}{\mddefault}{\updefault}$2n$}}}}
\put(6335,-3554){\rotatebox{335.0}{\makebox(0,0)[lb]{\smash{{\SetFigFont{7}{8.4}{\rmdefault}{\mddefault}{\updefault}$\dots$}}}}}
\put(6486,-4058){\makebox(0,0)[lb]{\smash{{\SetFigFont{7}{8.4}{\rmdefault}{\mddefault}{\updefault}$\dots$}}}}
\put(6314,-4676){\rotatebox{25.0}{\makebox(0,0)[lb]{\smash{{\SetFigFont{7}{8.4}{\rmdefault}{\mddefault}{\updefault}$\dots$}}}}}
\end{picture}%
\caption{A configuration in $\cal{SC}(n,m;\op)$}
\label{Swiss_disc}
\end{figure}

The 2-collection $\cal{SC}$ is defined as follows. For $m,n \geqslant 0$ with $m+n > 0$, $\cal{SC}(n,m;\op)$
is the configuration space defined above and $\cal{SC}(0,0;\op) = \emptyset$. For $n \geqslant 0$,  $\cal{SC}(n,0;\cl)$
is defined as $\cal D_2(n)$ and $\cal{SC}(n,m;\cl) = \emptyset$ for $m \geqslant 1$. 
The 2-colored operad structure in $\cal{SC}$ is given, as usual, by insertion of disks.

There is a suboperad $\SCvor$ of $\SC$ defined by $\SCvor(n,m;x) = \SC(n,m;x)$, if $x = \cl$ or $m \geqslant 1$ and by
$\SCvor(n,m;x) = \emptyset$, otherwise.  
The above definition says that $\SCvor$ coincides with $\SC$ except for $m=0$ and
$x=\op$, where $\SCvor(n,0,\op)=\emptyset$ for any $n \geqslant 0$. 
The operad $\SCvor$ is equivalent to the one defined by Voronov in \cite{Voronov99}, while 
$\SC$ coincides with the one defined by Kontsevich in \cite{Kontse99}.

\medskip

\noindent{\bf Notation.} The homology of $\SC$ is denoted by $\hsc$ while that of $\SCvor$ is denoted
by $\hscvor$.

\subsection{Conventions and notations.}

\subsubsection{Generators}
In the paper we will have specific generators in the different operads considered, mainly two families of elements.

The first family is
$\{f_2,g_2,e_{0,2},e_{1,1},e_{1,0}\}$
and the second family is
$\{\mathfrak l_2,\mathfrak c_2,\mathfrak n_{0,2},\mathfrak n_{1,1},\mathfrak n_{1,0}\}.$

The following array sum up the properties of the elements.
The array must be read as follows: $f_2\in M(\cl,\cl;\cl)$ means that it is an operation on two closed variables giving a closed variable;
the representation is $\kfield$, that is, $f_2$ is a symmetric operation. The degree is $0$.

$$\begin{array}{|c|c|c|c|c|c|}
\hline
\text{element}&f_2&g_2&e_{0,2}&e_{1,1}&e_{1,0}\\
\text{color}& M(\cl,\cl;\cl)& M(\cl,\cl;\cl)&M(\op,\op;\op)&M(\cl,\op;\op)&M(\cl;\op)\\
\text{representation}&\kfield &\kfield & \kfield[S_2] & \kfield[S_2] \text{\ in\ } M(\cl,\op;\op)\oplus M(\op,\cl;\op)&\kfield \\
\text{degree} &0& 1&0 &0 &0\\ 
\hline
\text{element}&\mathfrak l_2&\mathfrak c_2&\mathfrak n_{0,2}&\mathfrak n_{1,1}&\mathfrak n_{1,0}\\
\text{color}& M(\cl,\cl;\cl)& M(\cl,\cl;\cl)&M(\op,\op;\op)&M(\cl,\op;\op)&M(\cl;\op)\\
\text{representation}&\sgn_2 &\sgn_2 & \kfield[S_2] &  \kfield[S_2] \text{\ in\ } M(\cl,\op;\op)\oplus M(\op,\cl;\op)&\kfield \\
\text{degree} & 0 & -1 & 0 & 0 & -1 \\ 
\hline
\end{array}$$

\medskip

Given elements $\{x_1,\ldots,x_n\}$ with specific colors, representation and degrees, the $\mathbb S$-module $<x_1,\ldots,x_n>$ is the 
$\mathbb S$-module generated by these elements, with the action of the symmetric group indicated by the representation of the elements.
For example $<e_{1,1}>$ is the $\mathbb S$-module $M$ where $M(\cl,\op;\op)=\kfield e_{1,1}$, $M(\op,\cl;\op)=\kfield e_{1,1}\cdot (21)$ and is zero elsewhere.

\subsubsection{Notation for operads}

The operad $\Ger$, whose algebras are Gerstenhaber algebras is the operad 
$\cal F(E_\Ger,R_\Ger)$ with $E_\Ger=<f_2,g_2>$ and $R_\Ger$ is the space of relations given by
\begin{equation*}
\begin{aligned}
f_2(\ide\otimes f_2)&=f_2(f_2\otimes \ide), \\
g_2(g_2\otimes\ide)\cdot ((123)+(231)+(312))&=0, \\
g_2(\ide\otimes f_2)&=f_2(g_2\otimes\ide)+f_2(\ide\otimes g_2)\cdot (213).
\end{aligned} 
\end{equation*}

The suboperad generated by $f_2$ is the operad $\Com=\cal F(<f_2>,R_{\Com})$ where $R_\Com$ is the first relation. The suboperad generated by $g_2$ is the operad $\Lambda^{-1}\Lie$.

\medskip

The Koszul dual of the operad $\Ger$ is $\Ger^!=\Lambda\Ger$ (see e.g. \cite{GetJon94}). It is described as
$\cal F(E_{\Lambda\Ger},R_{\Lambda\Ger})$ with $E_{\Lambda\Ger}=<\mathfrak l_2,\mathfrak c_2>$ and $R_{\Lambda\Ger}$ is the space of relations given by
\begin{equation*}
\begin{aligned}
\mathfrak c_2(\ide\otimes \mathfrak c_2)&=-\mathfrak c_2(\mathfrak c_2\otimes \ide), \\
\mathfrak l_2(\mathfrak l_2\otimes\ide)\cdot ((123)+(231)+(312))&=0, \\
\mathfrak l_2(\ide\otimes \mathfrak c_2)&=\mathfrak c_2(\mathfrak l_2\otimes\ide)+\mathfrak c_2(\ide\otimes \mathfrak l_2)\cdot (213).
\end{aligned} 
\end{equation*}
The suboperad generated by $\mathfrak l_2$ is the operad $\Lie=\cal F(<\mathfrak l_2>,R_{\Lie})$ where $R_\Lie$ is the second relation. 
The suboperad generated by $\mathfrak c_2$ is the operad $\Lambda\Com$.

\medskip

The operad $\Ass$ is described as $\cal F(<e_{0,2}>,R_\Ass)$ where $R_\Ass$ is the relation 
$e_{0,2}(\ide\otimes e_{0,2})=e_{0,2}(e_{0,2}\otimes\ide).$ Note that we also use this notation replacing $e_{0,2}$ by $\mathfrak n_{0,2}$.

\section{The homology operads $\hscvor$ and $\hsc$}
We prove in this section that the homology operad $\hscvor$  is a quadratic Koszul operad and that the homology operad $\hsc$ is a quadratic-linear Koszul operad, extending the results obtained for the zero-th homology of $\SCvor$ and $\SC$ in \cite{HoeLiv12}.

\subsection{The operad $\hscvor$ is Koszul} Recall the Theorem

\begin{thm}[A. Voronov \cite{Voronov99}]
 An algebra over $\hscvor$ is a pair $(G,A)$, where $G$ is a Gerstenhaber algebra and $A$ is an associative algebra over the commutative ring $G$. 
\end{thm}

An algebra over the commutative ring $G$ corresponds to a degree $0$ map $\lambda : G \otimes A \to A$ satisfying:
\begin{equation*}
\lambda(cc',a)=\lambda(c,\lambda(c',a))=(-1)^{|c||c'|}\lambda(c',\lambda(c,a)) \ \mbox{ and }\  \lambda(c,aa')=\lambda(c,a)a'=(-1)^{|a||c|} a\lambda(c,a').
\end{equation*}

As a consequence

\begin{cor}\label{C:opscvor}The operad $\hscvor$ has a quadratic presentation $\mathcal F(E_v,R_v)$ where
$$E_v=<f_2,g_2,e_{0,2},e_{1,1}>$$ and $R_v$ is the sub-$\mathbb S$-module of $\mathcal F^{(2)}(E_v)$ generated by the relations:
\begin{itemize}
\item $R_{\Ger}$, for the Gerstenhaber structure defined by $f_2$ and $g_2$ and $R_\Ass$ for the associativity of $e_{0,2}$;
\item  $e_{1,1}$ is an action:

 $e_{1,1}(\ide\otimes e_{1,1})=e_{1,1}(f_2\otimes\ide),$

$e_{1,1}(\ide\otimes e_{0,2})=e_{0,2}(e_{1,1}\otimes \ide)=e_{0,2}(\ide\otimes e_{1,1})\cdot (213)$.
\end{itemize}

\end{cor}

\begin{lem}\label{L:scvor} Algebras over the Koszul dual operad $(\hscvor)^!$ of $\hscvor$ are of the form $(H,A,\rho)$ where $(H,[,],\times)$ is a $\Lambda\Ger$-algebra, $A$ is an associative algebra, and $\rho:H\otimes A\rightarrow A$ is a map of degree $0$ that satisfies the relations
\begin{equation}\label{E:scvordual}
\begin{aligned}
\rho([h,h'],a)=&\rho(h,\rho(h',a))-(-1)^{|h||h'|} \rho(h',\rho(h,a))\\
\rho(h,a\cdot a')=&\rho(h,a)\cdot a'+(-1)^{|a||h|}a\cdot \rho(h,a')\\
\rho(h\times h',a)=&0
\end{aligned}
\end{equation}
\end{lem}

Note that the first two equations indicate that the map induced by $\rho$ from $H$ to $\End(A)$ has values in $\Der(A)$ and is a morphism of Lie algebras.

\begin{proof} Because $\hscvor$ has a binary quadratic presentation, we can use the direct computation of its Koszul dual operad presented in section \ref{S:Koszul}. 
Let us denote by $(\mathfrak l_2,\mathfrak c_2,\mathfrak n_{0,2},\mathfrak n_{1,1})$  the dual basis of 
$(f_2,g_2,e_{0,2},e_{1,1})$ in $E_v^{\vee}$. The degree of $\mathfrak c_2$ is $-1$ and all the others elements have degree $0$.

The Koszul dual operad of $\hscvor$ is
$(\hscvor)^!=\cal F(E_{v}^\vee)/(R_{v}^{\bot})$. 
 The pairing between $E_v$ and $E_{v}^{\vee}$ induces a pairing 
between $\cal F^{(2)}(E_{v})$ and 
$\cal F^{(2)}(E_v^{\vee})$.  One gets that $R_{v}^{\bot}(\cl,\cl,\cl;\cl)$ is the ideal defining $\Ger^!$, that is $R_{\Lambda\Ger}$.
Similarly $R_{v}^{\bot}(\op,\op,\op;\op)$ is the orthogonal of the associativity relation for $e_{0,2}$, that is, the associativity relation for $\mathfrak n_{0,2}$.

The space $\cal F(E_{v})(\cl,\cl,\op;\op)_0$ has dimension 3 and $R_{v}(\cl,\cl,\op;\op)_0$ has dimension 1.
As a consequence, the dimension of $R_{v}^{\bot}(\cl,\cl,\op;\op)_0$ is 2 and corresponds to the first relation. 

The space 
$\cal F(E_{v})(\cl,\op,\op;\op)$ has dimension 6 and $R_{v}(\cl,\op,\op;\op)$ has dimension 2.
Hence the dimension of $R_{v}^{\bot}(\cl,\op,\op;\op)$ is 4 and corresponds to the second relation.

The space $\cal F(E_{v})(\cl,\cl,\op;\op)_1$ has dimension 1 and $R_{v}(\cl,\cl,\op;\op)_1$ has dimension 0.
As a consequence, the dimension of $R_{v}^{\bot}(\cl,\cl,\op;\op)_{-1}$ is 1 and corresponds to the third relation. 
\end{proof}

In terms of generators and relations, it expresses as:

\begin{cor}\label{C:scvordual} The operad $(\hscvor)^!$ has a binary quadratic presentation $\cal F(E_{v^!},R_{v^!})$, where
$$E_{v^!}=< \mathfrak l_2, \mathfrak c_2, \mathfrak n_{0,2}, \mathfrak n_{1,1}>$$
and  $R_{v^!}$ is the sub-$\mathbb S$-module of $\mathcal F^{(2)}(E_{v^!})$ generated by the relations:
\begin{itemize}
\item $R_{\Lambda\Ger}$, for the $\Lambda\Ger$-structure defined by $\mathfrak l_2$ and $\mathfrak c_2$ and $R_\Ass$ for the
associativity of $\mathfrak n_{0,2}$;
\item  relations for $\mathfrak n_{1,1}$ :

 $\mathfrak n_{1,1}(\mathfrak l_2\otimes\ide)=\mathfrak n_{1,1}(\ide\otimes \mathfrak n_{1,1})\cdot( \ide-(213)), $

$\mathfrak n_{1,1}(\ide\otimes \mathfrak n_{0,2})=\mathfrak n_{0,2}(\mathfrak n_{1,1}\otimes \ide)+\mathfrak n_{0,2}(\ide\otimes \mathfrak n_{1,1})\cdot (213)$,

$\mathfrak n_{1,1}(\mathfrak c_2\otimes\ide)=0$.
\end{itemize}

\end{cor}

\begin{thm}The operad $\hscvor$ is Koszul.
\end{thm}
 
\begin{proof} In order to prove that $\hscvor$ is Koszul, we prove that $(\hscvor)^!$ is Koszul, using the rewriting method explained in \cite{LodVal}, and using a part of the computation made by Alm in \cite[AppendixA]{Alm11}. Recall that an algebra over $(\hscvor)^!$ is given by the following data:
\begin{itemize}
\item A  $\Lambda\Ger$-algebra $H$. We denote by $[x_1,x_2]$ the degree $0$ bracket and by $x_1\times x_2$ the degree $-1$ product. 
\item An associative algebra $A$. We denote by $a_1\cdot a_2$ the degree $0$ product.
\item A map $\rho: H\otimes A\rightarrow A$. We denote by $x\bullet a$ the element $\rho(x,a)$.
\end{itemize} 
The rewriting rules are
\begin{align}
(a_1\cdot a_2)\cdot a_3 &\mapsto a_1\cdot (a_2\cdot a_3) \nonumber \\
(x_1\times x_2)\times x_3 & \mapsto -x_1\times (x_2\times x_3)\nonumber \\
[[x_1,x_2],x_3] &\mapsto -[[x_2,x_3],x_1]-[[x_3,x_1],x_2] \nonumber \\
[x_1,x_2\times x_3] & \mapsto [x_1,x_2]\times x_3+x_2\times [x_1,x_3] \nonumber\\
x_1\bullet(a_1\cdot a_2) &\mapsto (x_1\bullet a_1)\cdot a_2 + a_1\cdot(x_1\bullet a_2) \nonumber\\
(x_1\times x_2)\bullet a & \mapsto  0 \nonumber\\
[x_1,x_2] \bullet a_1 &\mapsto x_1\bullet(x_2\bullet a_1) - x_2\bullet(x_1\bullet a_1). \nonumber
\end{align}
In order to study the confluence of critical monomials, it is enough to study the one involving both $x's$ and $a's$ because the one involving only $a's$ corresponds to the computation for the operad $\Ass$, and the one involving only $x's$ corresponds to the computation for the operad $\Lambda\Ger$. We know that a way to prove the Koszulity of these 2 operads is precisely to use the confluence of the critical monomials.

Hence the critical monomials left are
 \(x_1\bullet ((a_1\cdot a_2)\cdot a_3)\), \([x_1,x_2] \bullet(a_1\cdot a_2)\), \([[x_1,x_2],x_3] \bullet a_1\), \((x_1\times x_2)\bullet(a_1\cdot a_2)\), \( ((x_1\times x_2)\times x_3)\bullet a\) and
 \([x_1,x_2\times x_3]\bullet a\). The first  three have been proven to be confluent by J. Alm. The fourth critical monomial can be rewritten either as
\begin{align}
(x_1\times x_2) \bullet(a_1\cdot a_2) &\mapsto ((x_1\times x_2)\bullet a_1)\cdot a_2+a_1\cdot ((x_1\times x_2)\bullet a_2) \nonumber \\
&\mapsto 0 \nonumber
\end{align}
or $(x_1\times x_2) \bullet(a_1\cdot a_2) \mapsto 0$. The same is true for the fifth critical  monomial.

The critical monomial \([x_1,x_2\times x_3]\bullet a\) can be rewritten either as

\begin{align}
[x_1,x_2\times x_3 ] \bullet a &\mapsto x_1\bullet((x_2\times x_3)\bullet a)-(x_2\times x_3)\bullet(x_1\bullet a) \nonumber \\
&\mapsto 0 \nonumber
\end{align}
or
\begin{align}
[x_1,x_2\times x_3 ] \bullet a &\mapsto  ([x_1,x_2]\times x_3)\bullet a + (x_2\times [x_1,x_3])\bullet a \nonumber\\
&\mapsto 0 \nonumber 
\end{align}

Hence, all the critical monomials are confluent and $(\hscvor)^!$ is Koszul. As a consequence $\hscvor$ is a Koszul operad.
\end{proof}

\subsection{The operad $\hsc$ is Koszul}
In this section we follow closely the article by Imma Galvez-Carrillo, Andy Tonks and Bruno Vallette
\cite{GTV12} and our paper \cite{HoeLiv12} in order to prove that the homology operad $\hsc$ is Koszul. 


Recall from the computation of F. Cohen and A. Voronov and from \cite{HoeLiv12} that

\begin{prop}An $\hsc$-algebra (G,A,f) is a Gerstenhaber algebra $G$ and an associative algebra $A$ together with a central morphism of associative algebras $f:G\rightarrow A$. 
\end{prop}

\begin{cor}\label{C:hsccubic} The operad $\hsc$ has a presentation of the form $\cal F(E')/(R')$ where
\[ 
E'=<f_2,g_2,e_{0,2},e_{1,0}>
\] 
and the space of relations $R'$ is the sub-$\mathbb S$-module of $\mathcal F^{(2)}(E)\oplus  \mathcal F^{(3)}(E)$
defined by the relations
\begin{itemize}
\item $R_\Ger$ for the Gerstenhaber structure induced by $f_2$ and $g_2$ and $R_\Ass$ for the associativity of $e_{0,2}$;
\item  centrality of $e_{1,0}$: $e_{0,2}(e_{1,0}\otimes\ide)=e_{0,2}(\ide\otimes e_{1,0})\cdot (21)$;
\item a quadratic-cubical relation:  $e_{1,0}(f_2)=f_2(e_{1,0}\otimes e_{1,0}).$
\end{itemize}
\end{cor}

This corollary shows clearly that this presentation is quadratic and cubic. In order to apply the theory of \cite{GTV12}, one needs a presentation which is quadratic and linear. However, we will see in Proposition \ref{P:final} that the quadratic operad $\cal F(E')/(qR')$ obtained by killing the cubical elements in the relations of $R'$ plays also an important role for the study of $\hsc$.

The idea to obtain a presentation with quadratic-linear relations of $\hsc$ is 
 to add a new generator, in order to replace the quadratic-cubical relation by quadratic-linear relations. This new generator $e_{1,1}$,
will correspond at the level of algebras to the operation $\lambda(c,a):=f(c)a$.
Consequently, we introduce new relations in the operad corresponding to the relations  $f(c)a=af(c)=\lambda(c,a)$ and
$\lambda(c,f(c'))=f(cc')=f(c)f(c')$, that are present in the algebra setting.

Recall the theory explained  in \cite{GTV12} for quadratic-linear operads. A {\sl quadratic-linear} operad is of the form $\mathcal F(E)/(R)$ with $R\subset \mathcal F^{(1)}(E)\oplus  \mathcal F^{(2)}(E)$.
Such an $R$ is called quadratic-linear. We also ask the presentation to satisfy:
\[
\mbox{ (ql1): }R\cap E=\{0\} \quad \mbox{ and } \quad
\mbox{ (ql2): }(R\otimes E+E\otimes R)\cap \mathcal F^{(2)}(E) \subset R\cap \mathcal F^{(2)}(E).
\]

\begin{prop}\label{P:preshsc}
The operad $\hsc$ has a presentation $\mathcal F(E)/(R)$, where
\[ 
E=<f_2,g_2,e_{0,2},e_{1,1},e_{1,0}>
\] 
and the space of relations $R$ is the sub-$\mathbb S$-module of $\mathcal F^{(1)}(E)\oplus  \mathcal F^{(2)}(E)$
defined by $R=R_v\oplus R(e_{1,0})$, where $R_v$ is the space of quadratic relations of $\hscvor$ and $R(e_{1,0})$ is the sub-$\mathbb S$-module of $\cal F(E)$ generated by the following relations:
\begin{itemize}
\item two quadratic-linear relations:  $e_{1,1}=e_{0,2}(e_{1,0}\otimes\ide)$ and $e_{1,1}=e_{0,2}(\ide\otimes e_{1,0})\cdot (21)$,
\item a new quadratic relation: $e_{1,1}(\ide\otimes e_{1,0})=e_{1,0}(f_2).$
\end{itemize}
Moreover this presentation satisfies \emph{(ql1)} and \emph{(ql2)}.
\end{prop}

Here we recall the definition of a Koszul quadratic-linear operad given in \cite{GTV12}.

\begin{defn}  Let $q$ denote the projection $\mathcal F(E)\epi \mathcal F^{(2)}(E)$ and let $qR$ be the image of $R$ under this projection.
A quadratic-linear operad $\mathcal P=\mathcal F(E)/(R)$ satisfying (ql1) and (ql2) is said to be {\sl Koszul} if
$q\mathcal P:=\mathcal F(E)/(qR)$ is a quadratic Koszul operad. Its Koszul dual cooperad is 
$(\mathcal P)^{\ac}=((q\mathcal P)^{\ac},\partial_\varphi)$ where the differential $\partial_{\varphi}$ depends on the quadratic-linear relations.
\end{defn}

In the case of $\hsc$ presented as in Proposition \ref{P:preshsc},
the projection of $R=R_v\oplus R(e_{1,0})$ onto $\mathcal F^{(2)}(E)$ is 
$qR=R_v\oplus qR(e_{1,0})$, where $qR(e_{1,0})$ is the sub-$\mathbb S$-module of  $\mathcal F^{(2)}(E)$
generated by the relations
$0=e_{0,2}(e_{1,0}\otimes\ide)$, $0=e_{0,2}(e_{1,0}\otimes\ide)$ and $e_{1,1}(\ide\otimes e_{1,0})=e_{1,0}(f_2).$

Consequently a $q\hsc$-algebra  is an $\hscvor$-algebra $(G,A,\lambda)$ endowed with a degree $0$ linear map $f:G\rightarrow A$ satisfying $f(c)a=af(c)=0$  and $\lambda(c,f(c'))=f(cc')$ for all $c,c'\in G, a\in A$.
As in \cite{HoeLiv12}, the operad $q\hsc$ is obtained as the result of a distributive law between the operad $\hscvor$ and $\cal F(e_{1,0})$.
The distributive law is given by
\begin{equation}\label{E:distributive}
\begin{array}{ccc}
 \hscvor\circ \cal F(e_{1,0}) &\rightarrow &\cal F(e_{1,0})\circ \hscvor \\
e_{0,2}(e_{1,0}\otimes\ide),\ e_{0,2}(e_{1,0}\otimes\ide) &\mapsto& 0 \\
e_{1,1}(\ide\otimes e_{1,0})& \mapsto & e_{1,0}(f_2).\\
\end{array}
\end{equation}

\begin{prop}\label{P:qhsc} The operad $q\hsc$ is identical to the operad $\cal F(e_{1,0})\circ \hscvor$, with composition given by the distributive law (\ref{E:distributive}).
\end{prop}

\begin{thm}\label{T:SCKoszul} The operads $q\hsc$ and $\hsc$ are Koszul operads.
\end{thm}

\begin{proof} From  \cite[Chapter8]{LodVal}, one has that $q\hsc=\cal F(e_{1,0})\circ \hscvor$ is Koszul since $\hscvor$ and $\cal F(e_{1,0})$ are Koszul colored operads. By definition, it means that $\hsc$ is a quadratic-linear Koszul operad.
\end{proof}

\subsection{Description of the Koszul dual operad $(\hsc)^!$ of $\hsc$.}
In Proposition \ref{P:qhsc}, we have described $q\hsc$ as a distributive law between $\hscvor$ and $\cal F(e_{1,0})$. As a consequence
$(q\hsc)^!=(\hscvor)^! \circ \cal F(e_{1,0})^!$, with the operad structure given by the signed dual of the distributive law (\ref{E:distributive}). 
Recall from Corollary \ref{C:scvordual} that  $\{\mathfrak l_2,\mathfrak c_2,\mathfrak n_{0,2},\mathfrak n_{1,1}\}$ is the dual basis of $\{f_2,g_2,e_{2,0},e_{1,1}\}$ that generates
$E_{v^!}$. From relation (\ref{E:Koszuldual}), one has $\cal F(e_{1,0})^!=\cal F(\mathfrak n_{1,0})$ where $\mathfrak n_{1,0}$ has degree $-1$. The dual of the distributive law is given by

\begin{equation*}
\begin{array}{ccc}
\cal F(\mathfrak n_{1,0})\circ (\hscvor)^! &\rightarrow &(\hscvor)^!\circ \cal F(\mathfrak n_{1,0})\\
\mathfrak n_{1,0}(\mathfrak l_2) & \mapsto & \mathfrak n_{1,1}(\ide\otimes \mathfrak n_{1,0}) \cdot (\ide-(21))\\
\mathfrak n_{1,0}(\mathfrak c_2) & \mapsto & 0
\end{array}
\end{equation*}
Consequently, a $(q\hsc)^!$-algebra is an $(\hscvor)^!$-algebra $(H,A,\rho)$ satisfying conditions of  lemma \ref{L:scvor}, together with a linear map $\beta:H\rightarrow A$ of degree $-1$ satisfying 
\begin{equation}\label{E:beta}
\begin{aligned}
\beta([h,h'])=&(-1)^{|h|}\rho(h,\beta(h'))-(-1)^{|h||h'|+|h'|}\rho(h',\beta(h))\\
 \beta(h\times h')=&0.
 \end{aligned}
 \end{equation}

In order to understand the structure of an $(\hsc)^!$-algebra it is then enough to understand the differential on the operad $(q\hsc)^!$ that comes from the non-quadraticity of the operad $\hsc$.


Let $\varphi:qR\rightarrow E$ defined by
 \begin{equation*}
 \begin{cases}\varphi(e_{0,2}(e_{1,0}\otimes\ide))=\varphi(e_{0,2}(\ide\otimes e_{1,0})\cdot (21))=e_{1,1},\\
 \varphi(R_v)=0,\\ 
 \varphi(e_{1,1}(\ide\otimes e_{1,0})-e_{1,0}(f_2))=0.\end{cases}
\end{equation*}

The Koszul dual cooperad of $q\hsc$ is $(q\hsc)^{\ac}=C(sE,s^2qR)$, with the notation
of section \ref{S:Koszul}. To $\varphi$ is associated the composite map
$(q\hsc)^{\ac} \epi s^2qR\xrightarrow{s^{-1}\varphi}sE.$
There exists a unique
coderivation $\widetilde{\partial}_\varphi\, :\, (q\hsc)^{\ac}\rightarrow \cal F^c(sE)$  which extends this map. Moreover, $\widetilde{\partial}_\varphi$  induces a square zero coderivation $\partial_\varphi$ on
the Koszul dual cooperad $(q\hsc)^{\ac}$.
The Koszul dual  cooperad of $\hsc$ is by definition
$\hsc^{\ac}=(C(sE,s^2qR),\partial_\varphi).$


Recall from (\ref{E:Koszuldual}) that
$(q\hsc)^!=(\Lambda(q\hsc^{\ac}))^*$.
As a consequence, $\hsc^!=((q\hsc)^!,d_{\varphi})$, where $d_{\varphi}$ is obtained as a combination of transpose and signed suspension of $\partial_\varphi$. Namely, 
$\hsc^!$ is a differential graded operad and we have the following Proposition.

\begin{prop} 
An algebra over $\hsc^!$ consists in a dg $\Lambda\Ger$- algebra $(H,[,],\times, d_L)$, a dg associative algebra $(A,d_A)$, an action 
$\rho : H\otimes A \rightarrow A $
and a degree $-1$ map $\beta:H \rightarrow A$ such that, for all $l \in H, a \in A$, we have $d_A(\beta(l)) = -\beta(d_Ll)$ and that
the relations (\ref{E:scvordual}) and (\ref{E:beta}) are satisfied. Moreover, the following relation is satisfied:
\begin{equation}\label{E:diff}
d_A\rho(l,a)=\rho(d_Ll,a)+(-1)^{|l|}\rho(l,d_Aa)+\beta(l)a-(-1)^{|a|(|l|+1)}a\beta(l). 
\end{equation}
\end{prop}

Note that relation (\ref{E:diff}) says that the map $\beta : H \to A$ 
is central  up to homotopy having the map $\rho: H \otimes A \to A$
as the homotopy operator. 
For a geometrical description of the above relations in terms of the Kontsevich compactification \cite{kont:defquant-pub}, we refer the reader to \cite{Hoefel09,KS06b,HoeLiv12}.

This proposition translates at the level of operads in the following Corollary.

\begin{cor}The differential graded operad $(\hsc)^!$ has a  presentation $\cal F(E_{!},R_{!})$, where
$$E_{!}=< \mathfrak l_2, \mathfrak c_2, \mathfrak n_{0,2}, \mathfrak n_{1,1}, \mathfrak n_{1,0}>$$
and the vector space $R_{!}$ is the sub-$\mathbb S$-module of $\mathcal F^{(2)}(E_{!})$ generated by the relations:
\begin{itemize}
\item $R_{\Lambda\Ger}$, for the $\Lambda\Ger$-structure defined by $\mathfrak l_2$ and $\mathfrak c_2$ and $R_\Ass$ for the 
associativity of $\mathfrak n_{0,2}$;
\item  relations for $\mathfrak n_{1,1}$ :

 $\mathfrak n_{1,1}(\mathfrak l_2\otimes\ide)=\mathfrak n_{1,1}(\ide\otimes \mathfrak n_{1,1})\cdot( \ide-(213)), $

$\mathfrak n_{1,1}(\ide\otimes \mathfrak n_{0,2})=\mathfrak n_{0,2}(\mathfrak n_{1,1}\otimes \ide)+\mathfrak n_{0,2}(\ide\otimes \mathfrak n_{1,1})\cdot (213)$,

$\mathfrak n_{1,1}(\mathfrak c_2\otimes\ide)=0$;

\item relations for $\mathfrak n_{1,0}$:

$\mathfrak n_{1,0}(\mathfrak l_2)=\mathfrak n_{1,1}(\ide\otimes \mathfrak n_{1,0})\cdot ((12)-(21))$,

$\mathfrak n_{1,0}(\mathfrak c_2)=0$.

\end{itemize}
The differential is given by $d\mathfrak n_{1,1}=\mathfrak n_{0,2}(\mathfrak n_{1,0}\otimes\ide)-\mathfrak n_{0,2}(\ide\otimes\mathfrak n_{1,0})\cdot(21)$ and vanishes elsewhere.

\end{cor}

\subsection{On the homology of $\hsc^!$}

In \cite{HoeLiv12}, we have considered the zero-th homology operad of $\SC$. In particular, the description of $H_0(\SC)^!$ (\cite[Proposition 6.3.2]{HoeLiv12}) is
the following
\begin{prop}The differential graded operad $H_0(\SC)^!$ has a  presentation $\cal F(E_{0},R_{0})$, where
$$E_{0}=< \mathfrak l_2, \mathfrak n_{0,2}, \mathfrak n_{1,1}, \mathfrak n_{1,0}>$$
and the vector space $R_{0}$ is the sub-$\mathbb S$-module of $\mathcal F^{(2)}(E_{0})$ generated by the relations:
\begin{itemize}
\item $R_{\Lie}$, for the $\Lie$-structure defined by $\mathfrak l_2$ and $R_\Ass$ for the associativity of $\mathfrak n_{0,2}$;
\item  relations for $\mathfrak n_{1,1}$ :

 $\mathfrak n_{1,1}(\mathfrak l_2\otimes\ide)=\mathfrak n_{1,1}(\ide\otimes \mathfrak n_{1,1})\cdot( \ide-(213)), $

$\mathfrak n_{1,1}(\ide\otimes \mathfrak n_{0,2})=\mathfrak n_{0,2}(\mathfrak n_{1,1}\otimes \ide)+\mathfrak n_{0,2}(\ide\otimes \mathfrak n_{1,1})\cdot (213)$;

\item relations for $\mathfrak n_{1,0}$:

$\mathfrak n_{1,0}(\mathfrak l_2)=\mathfrak n_{1,1}(\ide\otimes \mathfrak n_{1,0})\cdot ((12)-(21))$.
\end{itemize}
The differential is given by $d\mathfrak n_{1,1}=\mathfrak n_{0,2}(\mathfrak n_{1,0}\otimes\ide)-\mathfrak n_{0,2}(\ide\otimes\mathfrak n_{1,0})\cdot(21)$ and is zero on all the other generators.
\end{prop}

From this, it is easy to prove the following Corollary.

\begin{cor} The dg operad $\hsc^!$ is the operad composite $\Lambda\Com\circ H_0(\SC)^!$ together with the distributive law given by
\begin{equation*}
\begin{array}{ccc}
 H_0(\SC)^!\circ \Lambda\Com &\rightarrow &\Lambda\Com\circ H_0(\SC)^! \\
\mathfrak l_2(\ide\otimes\mathfrak c_2) & \mapsto & \mathfrak c_2(\mathfrak l_2\otimes\ide)+\mathfrak c_2(\ide\otimes\mathfrak l_2)\cdot (213) \\
\mathfrak n_{1,1}(\mathfrak c_2\otimes\ide) &\mapsto & 0 \\
\mathfrak n_{1,0}(\mathfrak c_2) &\mapsto & 0 
\end{array}
\end{equation*}

\end{cor}

 As a consequence we get:

\begin{thm}\label{T:homology} Algebras over the homology of the operad $\hsc^!$ are triples $(H,A,\beta)$ where $H$ is a $\Lambda\Ger$-algebra, $A$
is an associative algebra and $\beta:H\rightarrow A$ is a central map of degree $-1$ satisfying $\beta(x\times y)=0$.
\end{thm}

\begin{proof}Recall from \cite[Theorem 7.2.5]{HoeLiv12} that algebras over the  operad $H_0(\SC)^!$ are triples $(L,A,f)$ where $L$ is a
Lie algebra, $A$
is an associative algebra and $f:L\rightarrow A$ is a central map of degree $-1$.
Using the K\"unneth formula for the plethysm product $\circ$ of $\mathbb S$-modules, as in \cite[Lemma 2.1.3]{Fr04}, we obtain that $H_*(\hsc^!)=
\Lambda\Com\circ H_*(H_0(\SC)^!)$ with the distributive law given by
\begin{equation*}
\begin{array}{ccc}
H_*( H_0(\SC)^!)\circ \Lambda\Com &\rightarrow &\Lambda\Com\circ H_*(H_0(\SC)^!) \\
{[}\mathfrak l_2](\ide\otimes\mathfrak c_2) & \mapsto & \mathfrak c_2([\mathfrak l_2]\otimes\ide)+\mathfrak c_2(\ide\otimes[\mathfrak l_2])\cdot (213) \\
{[}\mathfrak n_{1,0}](\mathfrak c_2) &\mapsto & 0 
\end{array}
\end{equation*}
where $[x]$ denotes the image of a cycle $x$ in $H_*(H_0(\SC)^!)$.
\end{proof}


\section{On the spectral sequence}

In this section we will show that the spectral sequence $E(\SC)$ associated to the stratification of the compactification of points in the upper half plane collapses at the second stage. We prove that, as an $\mathbb S$-module, $E^2(\SC)$  corresponds to the $\mathbb S$-module defined by $\hsc$, but prove that the operad structures are different.

\subsection{On the first sheet of the spectral sequence}
For this section, we refer to \cite{Voronov99}, \cite{Hoefel09} and \cite{Dolgushev11}.

For the compactification of points we are considering two different spaces: the space $C(n)$ of configurations of $n \geqslant 2$ points in the disk modded out by the action of the group of dilatations and translations, of dimension $2n-3$;
the space $C(n,m)$, with $2n+m \geqslant 2$, of configurations of $n$ points in the upper half plane, $m$ points on the line, modded out by the action of the group of dilatations
and translations along the line, of dimension $2n+m-2$.

The operad $\SC$ is homotopy equivalent to the Fulton Mac Pherson compactification of $C(n)$ for $\SC(n,0;\cl)$ and of $C(n,m)$ for $\SC(n,m;\op)$. Since $C(1)$ and $C(0,1)$ are not well defined, we introduce both $\SC(1,0;\cl)$ and $\SC(0,1;\op)$ as the one point spaces containing the identity element of the closed and open colors, respectively. It has been proven by Getzler and Jones in \cite{GetJon94}, that the filtration associated to the stratification of the compactification of $C(n)$ induces a spectral sequence $E(\cal D_2)$, whose first sheet coincides with the cobar construction of the cooperad $(\Ger)^{\ac}$. Furthermore, the spectral sequence collapses at the second stage,  and $E^2(\cal D_2)$ coincides, as an operad, with $\Ger$. We will then focus on the open part of the Swiss-cheese operad.
From \cite[Theorem 5.2]{Hoefel09}, there is a stratification of $\overline{C(n,m)}$ indexed by partially planar trees, which induces a topological filtration

$$F^p:=F^p(\overline{C})=\{\text{closure of the union of strata of dimension}\ p\}.$$
It yields a spectral sequence, whose first sheet is given by
$$E^1(\SC)_{p,q}=H_{p+q}(F^p,F^{p-1})=H^{-q}(F^p\setminus F^{p-1}).$$

Let $\mathbb T(n,m)_p$ be the set of partially planar trees with $n$ closed inputs, $m$ open inputs, the output being open and $v:=2n+m-p-1$  vertices. To any vertex $v_i$ of  a tree $T\in\mathbb T(n,m)_p$ is associated the triple $(n_i,m_i,x_i)$ corresponding respectively, to its closed, open inputs, and its output. One has the relation
\begin{equation}\label{E:vertexformula}
\begin{aligned}
\sum_{i=1}^v n_i=&n+\sum_{i=1}^v \delta_{x_i,\cl}, \\
\sum_{i=1}^v m_i=&m+\sum_{i=1}^v \delta_{x_i,\op}-1. 
\end{aligned}
\end{equation}
Since any tree $T\in \mathbb T(n,m)_p$ is responsible for a strata $C_T:=\prod\limits_{i=1}^v \SC(n_i,m_i;x_i)$
one gets that
\begin{equation}\label{E:E1}
E^1(\SC)(n,m)_{p,q}=\bigoplus_{\substack{T\in\mathbb T(n,m)_p,\\ u_1+\ldots+u_v=-q}} \bigotimes_{i=1}^v H^{u_i}(\SC(n_i,m_i;x_i)).
\end{equation}

Because we are not exactly using the notation of \cite{Dolgushev11}, we need the following Lemma.

\begin{lem}\label{lemma:cobar}
The operad $E^1(\SC)$ coincides with the cobar construction of the cooperad $(\Lambda_\cl\Lambda\hsc)^*$. More precisely, one has
$$E^1(\SC)(n,m)_{p,q}=\Omega((\Lambda_\cl\Lambda\hsc)^*)(n,m;\op)^{(2n+m-p-1)}_{p+q},$$
where the upper index corresponds to the weight grading by the number of vertices of the trees involved and the lower index corresponds to the total degree.
\end{lem}

\begin{proof}
Recall that for any cooperad $\cal C$, $\Omega(\cal C)$ is the free operad $\cal F(s^{-1}\overline{\cal C})$, where $\overline{\cal C}$ is the coaugmentation ideal of the cooperad. Since $E^1(\SC)$ is also a free 
$2$-colored operad, they have the same description in terms of trees. One has,
\begin{multline*}
H^u(\SC(n,m;x))=\hsc^*(n,m;x)_{-u}=(\Lambda^{-1}\hsc^*)(n,m;x)_{-u+n+m-1}\\
=(\Lambda_{\cl}^{-1}\Lambda^{-1}\hsc^*)(n,m;x)_{-u+2n+m-1-\delta_{x,\cl}} 
=(s^{-1}\Lambda_{\cl}^{-1}\Lambda^{-1}\hsc^*)(n,m;x)_{-u+2n+m-1-\delta_{x,\cl}-1}
\end{multline*}
Using the description of $E^1(\SC)$ in (\ref{E:E1}) and the formulas in (\ref{E:vertexformula}), one gets
$$\sum_{i=1}^v (-u_i+2n_i+m_i-1-\delta_{x_i,\cl}-1)=q+2n+m+\sum_{i=1}^v \delta_{x_i,\cl}+\delta_{x_i,\op}-2v-1=q+2n+m-v-1=q+p,$$
which explains the grading obtained.
From \cite{Hoefel09} and \cite{Dolgushev11}, we know that the differentials of the two operads coincide. As a consequence, the two differential graded operads coincide. \end{proof}

\subsection{On the second sheet of the spectral sequence}

Theorem \ref{T:SCKoszul} asserts that $\hsc$ is a Koszul operad, which expresses that
$$\Omega(\hsc^{\ac})\rightarrow \hsc$$ is a quasi-isomorphism of operads.
Since all the graded vector spaces involved are finite dimensional, there is a quasi-isomorphism of cooperads
$$\hsc^*\rightarrow (\Omega(\hsc^{\ac}))^*=B((\hsc^{\ac})^*)\stackrel{(\ref{E:Koszuldual})}{=}B(\Lambda(\hsc^!)).$$
Applying the bar-cobar adjuntion we have a sequence of quasi-isomorphisms: 
$$\Omega(\hsc^*)\rightarrow \Omega B(\Lambda(\hsc^!)) \to \Lambda(\hsc^!).$$
Now applying the functor $\Lambda_{\cl}^{-1}\Lambda^{-1}$ to the above morphism and using Lemma \ref{lemma:cobar}, we finally have the quasi-isomorphism:
\begin{equation}\label{E:E1vshsc}
E^1(\SC)=\Omega((\Lambda_\cl\Lambda\hsc)^*)\rightarrow \Lambda_\cl^{-1}(\hsc^!).
\end{equation}

\begin{prop}\label{P:final}
The operad $E^2(\SC)$ is the quadratic operad $\cal F(E',qR')$, where
\[ 
E'=<f_2,g_2,e_{0,2},e_{1,0}>
\] 
and the space of relations $qR'$ is the sub-$\mathbb S$-module of $\mathcal F^{(2)}(E')$
defined by the relations
\begin{itemize}
\item $R_\Ger$ for the Gerstenhaber structure induced by $f_2$ and $g_2$ and $R_\Ass$ for the associativity of $e_{0,2}$;
\item  centrality of $e_{1,0}$: $e_{0,2}(e_{1,0}\otimes\ide)=e_{0,2}(\ide\otimes e_{1,0})\cdot (21)$;
\item the quadratic relation:  $e_{1,0}(f_2)=0.$
\end{itemize}

Equivalently, algebras over the  operad $E^2(\SC)$
are triples $(G,A,f)$ where $G$ is a Gerstenhaber algebra, $A$ is an associative algebra, $f:G\rightarrow A$ is a central degree $0$ map satisfying
$f(gg')=0, \forall g,g'\in G$.

\end{prop}

Note that the operad obtained is exactly the quadratic operad associated to the quadratic-cubical presentation of the operad $\hsc$ of Corollary \ref{C:hsccubic}.

\begin{proof}
 The operad $E^2(\SC)$ is the homology of the dg operad $E^1(\SC)$. Thanks to the quasi-isomorphism (\ref{E:E1vshsc}), it is the homology of the operad $\Lambda^{-1}_\cl(\hsc^!)$. From the computation of the homology of $\hsc^!$ obtained in Theorem \ref{T:homology}, we get the result.
\end{proof}

\begin{thm}\label{T:ss} The spectral sequence $E(\SC)$ collapses at the second stage.
\end{thm}

\begin{proof} Proposition \ref{P:final} implies that, as an $\mathbb S$-module,  $E^2(\SC)(n,m;\op)=\Ger(n)\otimes\Ass(m)=H_*(\SC)(n,m;\op).$ Because the spectral sequence converges to the homology of $\SC$, and because the dimension of the second sheet is the dimension of the target, one gets that $E(\SC)$ collapses at the second stage.
\end{proof}

\def\cprime{$'$} \def\cprime{$'$} \def\cprime{$'$} \def\cprime{$'$}
\providecommand{\bysame}{\leavevmode\hbox to3em{\hrulefill}\thinspace}
\providecommand{\MR}{\relax\ifhmode\unskip\space\fi MR }
\providecommand{\MRhref}[2]{%
  \href{http://www.ams.org/mathscinet-getitem?mr=#1}{#2}
}
\providecommand{\href}[2]{#2}

\end{document}